\documentclass[11pt]{article}

\usepackage{amsmath}
\usepackage{amssymb}

 \newtheorem{thm}{Theorem}[section]
 \newtheorem{cor}[thm]{Corollary}
 
 \newtheorem{prop}[thm]{Proposition}
 \newtheorem{rem}[thm]{Remark}
 \numberwithin{equation}{section}
 \newenvironment{proof}{\medbreak\noindent{\it Proof:}\rm}{\hfill$\square$\rm}

\newcommand{\A}{{\mathcal  A}}
\newcommand{\C}{{\mathbb  C}}
\newcommand{\Cn}{{\mathbb  C}^n}
\newcommand{\Csn}{{\mathbb  C}_*^n}
\newcommand{\D}{{\mathcal  D}}
\newcommand{\R}{{\mathbb  R}}
\newcommand{\Rn}{{\mathbb  R}^n}
\newcommand{\M}{{\mathcal  M}}
\newcommand{\T}{{\mathcal  T}}

\newcommand{\AP}{{\operatorname{AP}}}
\newcommand{\HAP}{{\operatorname{HAP}}}
\newcommand{\Log}{{\operatorname{Log}\,}}
\newcommand{\codim}{{\operatorname{codim}\,}}
\newcommand{\Imm}{{\operatorname{Im}\,}}
\newcommand{\EE}{{\operatorname{E}\,}}
\newcommand{\supp}{{\rm supp}\,}

\author{Alexander Rashkovskii\thanks{Support by the Institut Mittag-Leffler (Djursholm, Sweden) is gratefully
acknowledged.}}
\date{}
\title{A remark on amoebas in higher codimensions}

\begin{document}

\maketitle

\begin{abstract}
It is shown that tube sets over amoebas of algebraic varieties of dimension $q$ in $\Csn$ (and, more generally, of almost periodic
holomorphic chains in $\Cn$) are $q$-pseudoconcave
in the sense of Rothstein. This is a direct consequence of a
representation of such sets as supports of positive closed currents.
\end{abstract}

\section{Introduction}
Let $V$ be an algebraic variety in $\Csn=(\C\setminus 0)^n$. Its
image $\A_V=\Log V$ under the mapping
$\Log(z_1,\ldots,z_n)=(\log|z_1|,\ldots,\log|z_n|)$ is called the
{\it amoeba} of $A$. The notion was introduced in \cite{GKZ} and has
found numerous applications in complex analysis and algebraic
geometry, see the survey \cite{Mi}.

The amoeba of $V$ is a closed set with non-empty complement
$\A_V^c=\Rn\setminus\A_V$. If $V$ is of codimension $1$, then each
component of $\A_V^c$ is convex because $\Log^{-1}(\A_V^c)$ is the 
intersection of a family of domains of holomorphy. This is no longer
true for varieties of higher codimension; nevertheless, some
rudiments of convexity do take place. As shown by Henriques
\cite{H}, if $\codim V=k$, then $\A_V^c$ is $(k-1)$-convex, a notion
defined in terms of homology groups for sections by $k$-dimensional
affine subspaces. A local result, due to Mikhalkin \cite{Mi}, states
that $\A_V$ has no {\it supporting $k$-cap}, i.e., a ball $B$ in a
$k$-dimensional plane such that $\A_V\cap B$ is nonempty and
compact, while $\A_V\cap (B+\epsilon v)=\emptyset$ for some
$v\in\Rn$ and all sufficiently small $\epsilon>0$.

The notion of amoeba was adapted by Favorov \cite{F} to zero sets of
holomorphic almost periodic functions in a tube domain as "shadows"
cast by the zero sets to the base of the domain; a precise
definition is given in Section~4. In \cite{FGS}, Henriques' result
was extended to amoebas of zero sets of so-called regular
holomorphic almost periodic mappings. This was done by a reduction
to the case considered in \cite{H} where the proof was given by
methods of algebraic geometry.

In this note, we propose a different approach to convexity
properties of amoebas in higher codimensions. It is purely
analytical and works equally well for both algebraic and almost
periodic situations. Moreover, we get our (pseudo)convexity results
as a by-product of a representation of an amoeba as the support of a
certain natural measure determined by the "density" of the zero set.

Let us start with a hypersurface case. When
$V=\{P(z)=0\}\subset\Csn$ is defined by a Laurent polynomial $P$,
the function
$$N_P(y)=\frac1{(2\pi)^n}\int_{[-\pi,\pi]^n}\log|P(e^{y_1+i\theta_1},\ldots,
e^{y_n+i\theta_n})|\,d\theta,$$ known in tropical mathematics
community as the {\it Ronkin function}, is convex in $\Rn$ and linear
precisely on each connected component of $\A_V^c$. This means that
the support of the current $dd^c N_P(\Imm z)$ equals
$T_{\A_V}=\Rn+i\A_V$, the tube set in $\Cn$ with base $\A_V$. Since
the complement to the support of a positive closed current of
bidegree $(1,1)$ is pseudoconvex (being a domain of existence for a
pluriharmonic function), this implies pseudoconvexity of
$T_{\A_V^c}$ and thus convexity of every component of $\A_V^c$. Of
course, the function $N_P$ gives much more than simply generating
the amoeba (see, for example, \cite{R3}, \cite{FPS}, \cite{PR}).

The same reasoning applies to amoebas of holomorphic almost periodic functions $f$ with Ronkin's function $N_P$ replaced by the mean value $\M_f(y)$ of $\log|f|$ over the real planes $\{x+iy:\: x\in\Rn\}$, $y\in\Rn$.

What we will do in the case of codimension $k>1$, is presenting
$T_{\A_V}$ as the support of a closed positive current of bidegree
$(k,k)$ (namely, a mean value current for the variety or, more
generally, for a holomorphic chain) and then using a theorem on
$(n-k)$-pseudoconcavity, in the sense of Rothstein, of supports of
such currents due to Fornaess and Sibony \cite{FS}. In addition, we
show that for a closed set $\Gamma\subset\Rn$, Rothstein's
$(n-k)$-pseudoconcavity of $T_{\Gamma}$ implies the absence of
$k$-supporting caps of the set $\Gamma$ (Proposition~\ref{prop:1}).

We obtain our main result, Theorem~\ref{theo:amch}, for arbitrary
almost periodic holomorphic chains, which is a larger class than
zero sets of regular almost periodic holomorphic mappings, and the
situation with algebraic varieties (Corollary~\ref{cor:1}) is a
direct consequence. The existence of the mean value currents was established in
\cite{FRR2}, so here we just combine it together with the theorem on
supports of positive closed currents. In this sense, this note is
just a simple illustration of how useful the mean value currents
are.

\section{Rothstein's $q$-pseudoconvexity}

We will use the following notion of $q$-pseudoconvexity, due to
W.~Rothstein \cite{Ro}, see also \cite{Ri}. Given $0<q<n$ and
$\alpha,\beta\in (0,1)$, the set
$$H=\{(z,w)\in\C^{n-q}\times\C^{q}:\: \|z\|_\infty<1,
\ \|w\|_\infty<\alpha\ {\rm or\ }
 \beta<\|z\|_\infty<1,\ \|w\|_\infty<1\}$$ is called an
$(n-q,q)$-Hartogs figure; here $\|z\|_\infty=\max_j|z_j|$. Note that
its convex hull $\hat H$ is the unit polydisc in $\C^n$. An open
subset $\Omega$ of a complex $n$-dimensional manifold $M$ is said to
be {\it $q$-pseudoconvex} in $M$ if for any $(n-q,q)$-Hartogs figure
$H$ and a biholomorphic map $\Phi:\hat H\to M$, the condition
$\Phi(H)\subset\Omega$ implies $\Phi(\hat H)\subset\Omega$. If this
is the case, we will also say that $M\setminus\Omega$ is {\it
$q$-pseudoconcave} in $M$.

Loosely speaking, the $q$-pseudoconvexity is the Kontinuit\"atssatz
with respect to $(n-q)$-polydiscs; usual pseudoconvexity is
equivalent to $(n-1)$-pseudoconvexity.

\begin{thm}\label{theo:FS} {\rm (\cite{FS}, Cor. 2.6)} The support
of a positive closed current of bidimension $(q,q)$ on a complex
manifold $M$ is $q$-pseudoconcave in $M$.
\end{thm}

It is easy to see that for tube sets, $(n-k)$-pseudoconcavity
implies absence of $k$-caps in the sense of Mikhalkin.

\begin{prop}\label{prop:1} Let $\Gamma$ be a closed subset of
a convex open set $D\subset\Rn$. If the tube set
$T_\Gamma=\Rn+i\Gamma$ is $(n-k)$-pseudoconcave in the tube domain
$T_D=\Rn+iD$, then $\Gamma$ has no $k$-supporting caps.
\end{prop}

\begin{proof} Assume $\Gamma$ has a $k$-supporting cap $B$. We
assume that the vector $v$ in the definition of the cap is
orthogonal to $B$ (in the general case, one gets an image of a
Hartogs figure under a non-degenerate linear transform). Choose
coordinates in $\Rn$ such that
$$B=\{(y',y'')\in \R^k\times\R^{n-k}:\: \|y'\|_\infty<1,\ y''=0\},$$
$$\{\beta<\|y'\|_\infty <1, \ \|y''\|_\infty<1\}\subset
D\setminus\Gamma,\quad \beta\in (0,1),$$ and $B+\epsilon v\subset
D\setminus \Gamma$ for all $\epsilon\in (0,1)$, where $v=(0,v'')$,
$\|v''\|_\infty=1$. Since $D$ is open, $\{y:\: \|y'\|_\infty <1, \
\|y''-\frac12 v''\|_\infty<\alpha\}\subset D\setminus\Gamma$ for
some $\alpha\in (0,1)$. Therefore, the $\frac{i}2v''$-shift of the
corresponding $(k,n-k)$-Hartogs figure $H$ is a subset of the tube
set $T_D\setminus T_\Gamma$. Since $B$ is a subset of the shifted
polydisc $\hat H+\frac{i}2v''$ and $B\cap\Gamma\neq\emptyset$, the
set $T_\Gamma$ is not $(n-k)$-pseudoconcave.
\end{proof}

\section{Almost periodic holomorphic chains}

Here we recall some facts from Ronkin's theory of holomorphic almost
periodic mappings and currents; for details, see \cite{R2},
\cite{R5}, \cite{FRR2}, \cite{FRR3}, and the survey \cite{FR}.

Let $\T_t$ denote the translation operator on $\R^n$ by $t\in\R^n$,
then for any function $f$ on $\R^n$, $({\T}_t^*f)(x)=f(\T_t
x)=f(x+t)$.

A continuous mapping $f$ from $\R^n$ to a metric space $X$ is called
{\it almost periodic} if the set $\{{\T}_t^*f\}_{t\in\R^n}$ is
relatively compact in $C(\R^n,X)$ with respect to the topology of
uniform convergence on $\R^n$. The collection of all almost periodic
mappings from $\R^n$ to $X$ will be denoted by $\AP(\R^n,X)$.

As is known from classical theory of almost periodic functions, any
function $f\in \AP(\R^n,\C)$ has its mean value $\M_f$ over $\R^n$,
$$\M_f=\lim_{s\to\infty}(2s)^{-n}\int_{\Pi_s} f\,dm_n,$$
where $\Pi_s=\{x\in\R^n:\: \|x\|_\infty<s\}$ and $m_n$ is the
Lebesgue measure in $\R^n$.

Let $D$ be a convex domain in $\R^n$, $T_D=\R^n+iD$. A continuous
mapping $f: T_D\to X$ is called {\it almost periodic on $T_D$} if
$\{{\T}_t^*f\}_{t\in\R^n}$ is a relatively compact subset of
$C(T_D,X)$ with respect to the topology of uniform convergence on
each tube subdomain $T_{D'},\,D'\Subset D$. The collection of all
almost periodic mappings from $T_D$ to $X$ will be denoted by
$\AP(T_D,X)$.

The set $\AP(T_D,\C)$ can be defined equivalently as the closure
(with respect to the topology of uniform convergence on each tube
subdomain $T_{D'},\,D'\Subset D$) of the collection of all
exponential sums with complex coefficients and pure imaginary
exponents (frequencies). The mean value of $f\in \AP(T_D,\C)$ is a
continuous function of $\Imm z$. The collection of all holomorphic
mappings $f\in \AP(T_D,\C^k)$ will be denoted by $\HAP(T_D,\C^k)$.
In particular, any mapping from $\C^n$ to $\C^k$, whose components
are exponential sums with pure imaginary frequencies, belongs to
$\HAP(\C^n,\C^k)$.

The notion of almost periodicity can be extended to distributions.
For example, a measure $\mu$ on $T_D$ is called almost periodic if
$\phi(t)=\int ({\T}_t)_*\phi\,d\mu\in \AP(\Rn,\C)$ for every
continuous function $\phi$ with compact support in $T_D$.
Furthermore, it can be extended to holomorphic chains as follows.

Let $Z=\sum_j c_jV_j$ be a holomorphic chain on $\Omega\subset\C^n$
supported by an analytic variety $|Z|=\cup_jV_j$ of pure dimension
$q$. Its integration current $[Z]$ acts on test forms $\phi$ of
bidegree $(q,q)$ with compact support in $\Omega$ (shortly,
$\phi\in\D_{q,q}(\Omega)$) as
$$ ([Z],\phi)=\int_{Reg|Z|}\gamma_Z\phi=\sum_j c_j\int_{Reg\,V_j}\phi,$$
where the function $\gamma_Z$ takes constant positive integer values
on the connected components of $Reg|Z|$. The $q$-dimensional volume
of $Z$ in a Borel set $\Omega_0\subset \Omega$ is
$$Vol_Z(\Omega_0)=\int_{\Omega_0\cap Reg|Z|}\gamma_Z\beta_q$$
(the mass of the trace measure of $[Z]$ in $\Omega_0$). If $f$ is a
holomorphic mapping on $\Omega$ such that $|Z|=f^{-1}(0)$ and
$\gamma_Z(z)$ equals the multiplicity of $f$ at $z$, the chain will
be denoted by $Z_f$.

A $q$-dimensional holomorphic chain $Z$ on $T_D$ is called an {\it
almost periodic holomorphic chain} if $(\T_{t}^*[Z],\phi)\in
AP(T_D,\C)$ for any test form $\phi\in\D_{q,q}(T_D)$. Here
$\T_{t}^*S=\sum\alpha_{IJ}(z+t)\,dz^I\wedge d\bar z^J$ is the
pullback of the current $S=\sum\alpha_{IJ}(z)\,dz_I\wedge d\bar
z_J$.

For any $f\in \HAP(T_D,\C)$, the chain (divisor) $Z_f$ is always
almost periodic; on the other hand, there exist almost periodic
divisors (starting already from dimension $n=1$) that are not
divisors of any holomorphic almost periodic function; when $n>1$,
even a periodic divisor need not be the divisor of a periodic
holomorphic function \cite{R4}. The situation with higher
dimensional mappings is even worse, since the chain $Z_f$ generated
by $f\in \HAP(T_D,\C^k))$, $k>1$, need not be almost periodic
\cite{FRR2}. It is however so if the mapping $f$ is {\it regular},
that is, if $\codim |Z_g|=k$ or $|Z_g|=\emptyset$ for every mapping
$g$ from the closure of the set $\{{\T}_t^*f\}_{t\in\R^m}$
\cite{FRR2}, \cite{FRR3}. A sufficient regularity condition
\cite{R2} shows that such mappings are generic.

Now we can turn to construction of the current that plays central
role in our considerations, the details can be found in \cite{FRR3}.
Let $Z$ be an almost periodic holomorphic chain of dimension $q$.
For any test form $\phi\in\D_{q,q}(T_D)$, the mean value
$\M_{\phi_Z}$ of the function $\phi_Z(t):=(\T_{t}^*[Z],\phi)\in
\AP(\R^n,\C)$ defines the {\it mean value current} $\M_Z$ of $Z$ by
the relation
$$(\M_Z,\phi)=\M_{\phi_Z}.$$ The current is closed and positive.
Since $\M_Z$ is translation invariant with respect to $x$, its
coefficients have the form $\M_{IJ}=m_n\otimes \M_{IJ}'$, where
$\M_{IJ}'$ are Borel measures in $D$. In addition, if $\psi=\sum
\psi_{IJ}dz_I\wedge d\bar z_J$ is a form with coefficients
$\psi_{IJ}\in\D(D)$ and $\chi_s$ is the characteristic function of
the cube $\Pi_s$, then there exists the limit
$$ \lim_{s\to\infty}(2s)^{-n}([Z],\chi_s\psi)=(\M_Z',\psi'),$$ where
$\M_Z'=\sum\M_{IJ}'dy_I\wedge dy_J$ and
$\psi'=\sum\psi_{IJ}dy_I\wedge dy_J$.

The trace measure $\mu_Z=\M_Z\wedge\beta_q$ can also be written as
$\mu_Z=m_n\otimes \mu_Z'$, where $\mu_Z'$ is a positive Borel
measure on $D$. The following result shows that it can be viewed as
a density of the chain $Z$ along $\Rn$.

\begin{thm}\label{theo:support} {\rm (\cite{FRR2}, \cite{FRR3})} Let $Z$
be an almost periodic holomorphic chain in a tube domain $T_D$. For
any open set $G\Subset D$ such that $\mu_Z'(\partial G)=0$, one has
$$ \lim_{s\to\infty}(2s)^{-n}Vol_Z(\Pi_s+iG)=\mu_Z'(G);$$
in addition, $\mu_Z'(G)=0$ if and only if $|Z|\cap T_{G}=\emptyset$.
\end{thm}

\begin{rem}\label{rem:1} {\rm For $Z=Z_f$ with regular $f \in \HAP(T_D,\C^k)$, Theorem
\ref{theo:support} was proved in \cite{R2} (for $k=n$) and \cite{Ra}
($k<n$), without using the notion of almost periodic chain. The
current $\M_{Z_f}$ can be constructed as follows. The coefficients
$a_{IJ}$ of the current $\log|f|(dd^c\log|f|)^{k-1}$ are locally
integrable functions on $T_D$, almost periodic in the sense of
distributions: $(\T_{t}^*a_{IJ},\phi)\in \AP(T_D,\C)$ for any test
function $\phi\in\D(T_D)$. Therefore, they possess their mean values
$A_{IJ}=\M_{a_{IJ}}$, and the current $\M_{Z_f}=dd^c(\sum
A_{IJ}dz_I\wedge d\bar z_J)$.}
\end{rem}

\section{Amoebas}

Following \cite{F}, if $Z$ is an almost periodic holomorphic chain
in $T_D$, then its {\it amoeba } $\A_Z$ is the closure of the
projection of $|Z|$ to $D$:
$$\A_Z=\overline{\Imm |Z|},$$
where the map $\Imm:\:\Cn\to \Rn$ is defined by
$\Imm(z_1,\ldots,z_n)=(\Imm z_1,\ldots,\Imm z_n)$. When $Z=Z_f$ for
a regular mapping $f\in \HAP(T_D,\C^p)$, we write simply $\A_f$.

Our convexity result is stated in terms of the tube set
$T_{\A_Z}=\Rn+i\A_Z$.

\begin{thm}\label{theo:amch}
If $Z$ is an almost periodic holomorphic chain of dimension $q$ in a
tube domain $T_D\subseteq\Cn$, then $T_{\A_Z}=\supp \M_Z$, where
$\M_Z$ is the mean value current of the chain $Z$. Therefore,
$T_{\A_Z}$ is $q$-pseudoconcave in $T_D$. In particular, for any
regular mapping $f \in \HAP(T_D,\C^k)$, the set $T_{\A_f}$ is
$(n-k)$-pseudoconcave.
\end{thm}

\begin{proof} By Theorem \ref{theo:support}, $\A_Z=\supp \mu_Z'$,
which can be rewritten as $$T_{\A_Z}=\supp m_n\otimes
\mu_Z'=\supp\M_Z.$$ Since the current $\M_Z$ is positive and closed,
Theorem~\ref{theo:FS} implies the corresponding pseudoconcavity.
\end{proof}

\medskip

This covers the algebraic case as well by means of the map
$\EE:\:\Cn\to\Csn$, $\EE(z_1,\ldots,z_n)=(e^{-i
z_1},\ldots,e^{-iz_n})$. For a Laurent polynomial $P$, the
exponential sum $\EE^*P$ is periodic in $T_{\Rn}$, and its mean
value $\M_{\log |E^*P|}$ coincides with the Ronkin function $N_P$.
Furthermore, given an algebraic variety $V\subset\Csn$, its pullback
$\EE^*V$ is almost periodic (actually, periodic) in $\Cn$ and
$\A_{\EE^*V}=\A_V$, which gives

\begin{cor}\label{cor:1} The set $T_{\A_V^c}$ for an algebraic variety
$V\subset\Csn$ of pure codimension $k$ is $(n-k)$-pseudoconvex.
\end{cor}

\vskip1cm

Tek/Nat, University of Stavanger, 4036 Stavanger, Norway

\vskip0.1cm

{\sc E-mail}: alexander.rashkovskii@uis.no


\begin{thebibliography}{11}

\bibitem{D1}
J.-P. Demailly, \textit{Monge-Amp\`ere operators, Lelong numbers and
intersection theory.} Complex Analysis and Geometry (Univ. Series in
Math.), ed. by V.~Ancona and A.~Silva, Plenum Press, New York, 1993,
115--193.

\bibitem{FGS} A. Fabiano, J. Guenot, J. Silipo, \textit{Bochner
transforms, perturbations and amoebae of holomorphic almost periodic
mappings in tube domains.} Complex Var. Elliptic Equ. \textbf{52}
(2007), no. 8, 709--739.

\bibitem{F}
S.Yu. Favorov, \textit{Holomorphic almost periodic functions in tube
domains and their amoebas.} Comput. Methods Funct. Theory {\textbf
1} (2001), no. 2, 403--415.

\bibitem{FRR2}
 S.Yu.\,Favorov, A.Yu.\,Rashkovskii and L.I.\,Ronkin, \textit{Almost
periodic currents and holomorphic chains.} C. R. Acad. Sci. Paris
\textbf{327}, Serie I (1998), 302--307.

\bibitem{FRR3}
S.Yu.\,Favorov, A.Yu.\,Rashkovskii and L.I.\,Ronkin, \textit{Almost
periodic currents, divisors and holomorphic chains.} Israel. Math.
Conf. Proc. \textbf{15} (2001), 67--88.

\bibitem{FR}
 S.Yu.\,Favorov and A.Yu.\,Rashkovskii, \textit{Holomorphic almost
 periodic functions.} Acta Appl. Math. \textbf{65} (2001), 217--235.

\bibitem{FS}
J.E. Fornaess and N. Sibony,  \textit{Oka's inequality for currents
and applications.} Math. Ann. \textbf{301} (1995), no. 3, 399--419.

\bibitem{FPS}
M. Forsberg, M. Passare, A. Tsikh,  \textit{Laurent determinants and
arrangements of hyperplane amoebas.}  Adv. Math. \textbf{151}
(2000), no. 1, 45--70.



\bibitem{GKZ}
I.M. Gelfand, M.M. Kapranov, A.V. Zelevinsky, \textit{Discriminants,
resultants, and multidimensional determinants.} Mathematics: Theory
\& Applications. Birkh\"auser Boston, Inc., Boston, MA, 1994.


\bibitem{H}
A. Henriques, \textit{An analogue of convexity for complements of
amoebas of varieties of higher codimension, an answer to a question
asked by B. Sturmfels.}  Adv. Geom. \textbf{4} (2004), no. 1,
61--73.


\bibitem{Mi}
G. Mikhalkin, \textit{Amoebas of algebraic varieties and tropical
geometry.} Different faces of geometry, 257--300, Int. Math. Ser.
(N. Y.), 3, Kluwer/Plenum, New York, 2004.

\bibitem{PR}
M. Passare and H. Rullg{\aa}rd,  \textit{Monge-Amp\`ere measures, and
triangulations of the Newton polytope.} Duke Math. J. \textbf{121}
(2004), no. 3, 481--507.

\bibitem {Ra}
A. Rashkovskii, \textit{Currents associated to holomorphic almost
periodic mappings.} Mat. Fizika, Analiz, Ceometria \textbf{2}
(1995), no. 2, 250--269.

\bibitem{Ri}
O. Riemenschneider, \textit{\"Uber den Fl\"acheninhalt analytischer
Mengen und die Erzeugung $k-$pseudokonvexer Gebiete.}  Invent. Math.
\textbf{2} (1967), 307--331.

\bibitem{R2}
L.I. Ronkin, \textit{Jessen's theorem for holomorphic almost
periodic mappings.} Ukrainsk. Mat. Zh. \textbf{42} (1990),
1094--1107.

\bibitem{R4}
L.I. Ronkin, \textit{Holomorphic periodic functions and periodic
divisors.} Mat. Fizika, Analiz i Geometria \textbf{2} (1995), no. 1,
108--122.

\bibitem{R5}
L.I. Ronkin, \textit{Almost periodic distributions and divisors in
tube domains.} Zap. Nauchn. Sem. POMI \textbf{247} (1997), 210--236.

\bibitem{R3}
L.I. Ronkin, \textit{On zeros of almost periodic functions generated
by functions holomorphic in a multicircular domain.}  Complex
analysis in modern mathematics, 239--251, FAZIS, Moscow, 2001.

\bibitem{Ro}
W. Rothstein, \textit{Ein neuer Beweis des Hartogsschen Hauptsatzes
und seine Ausdehnung auf meromorphe Funktionen.} Math. Z.
\textbf{53} (1950), 84--95.


\end{thebibliography}
\end{document}